\documentclass{article}

\usepackage[english]{babel}

\usepackage[letterpaper,top=2.5cm,bottom=2.5cm,left=3cm,right=3cm,marginparwidth=1.75cm]{geometry} 

\usepackage{amsmath, amssymb, amsthm} 
\usepackage{physics} 
\usepackage{mathrsfs} 
\usepackage{setspace} 
\usepackage{bm} 
\usepackage[normalem]{ulem} 
\usepackage{cancel} 
\usepackage{csquotes} 
\usepackage{biblatex} 
\usepackage[colorlinks=true, allcolors=blue]{hyperref} 
\usepackage{pgfplots} 
\usepackage{graphics} 
\usepackage{graphicx} 
\usepackage{tikz} 
\usepackage{tcolorbox} 
\usepackage{enumitem} 
\usepackage{subfiles} 
\usepackage{ytableau}
\usepackage{thmtools}
\usepackage{thm-restate}
\declaretheorem{lemma}

\hypersetup{colorlinks, citecolor=red, filecolor=black, linkcolor=blue, urlcolor=blue} 
\pgfplotsset{compat=1.17} 

\newcommand\s{\mathcal} 
\newcommand\ov{\overline} 
\newcommand\bb{\mathbb} 

\newcommand\NN{\bb{N}}

\newcommand\RR{\bb{R}}

\newcommand\ra{\Longrightarrow} 
\newcommand\sm{\setminus}

\title{Improved upper bounds on even-cycle creating Hamilton paths}
\author{John Byrne \thanks{University of Delaware, Department of Mathematical Sciences, \texttt{jpbyrne@udel.edu}} \and Michael Tait \thanks{Villanova University, Department of Mathematics \& Statistics, \texttt{michael.tait@villanova.edu}. Research partially supported by NSF grants DMS-2011553 and DMS-2245556.}}
\date{\today}
\bibliography{sample}

\begin{document}
\maketitle

\newtheorem{theorem}{Theorem}
\newtheorem{corollary}{Corollary}
\newtheorem{claim}{Claim}
\newtheorem{conjecture}{Conjecture}
\newtheorem{remark}{Remark}

\begin{abstract}
    We study the function $H_n(C_{2k})$, the maximum number of Hamilton paths such that the union of any pair of them contains $C_{2k}$ as a subgraph. We give upper bounds on this quantity for $k\ge 3$, improving results of Harcos and Solt\'esz, and we show that if a conjecture of Ustimenko is true then one additionally obtains improved upper bounds for all $k\geq 6$. {We also give bounds on $H_n(K_{2,3})$ and $H_n(K_{2,4})$. In order to prove our results, we extend a theorem of Krivelevich which counts Hamilton cycles in $(n, d, \lambda)$-graphs to bipartite or irregular graphs, and then apply these results to generalized polygons and the constructions of Lubotzky-Phillips-Sarnak and F\"uredi.}
\end{abstract}

\section{Introduction}

For a graph $G$, we denote by $H_n(G)$ the maximum number of Hamilton paths in $K_n$ the union of any two of which contains $G$ as a subgraph. The problem of studying $H_n(G)$ for various graphs $G$ is motivated by the fact that there is a natural correspondence between permutations of $[n]$ and directed Hamilton paths in $K_n$, and there has been extensive study on finding sets of permutations whose pairs satisfy some property (see for example \cite{borgsurvey} or \cite{EFP}). This allows problems involving permutations to be studied using graph-theoretic tools. An example of such a connection occurs in \cite{path_separation}. Cohen, Fachini, and K\"orner showed there that the maximum number of pairwise reversing permutations of $[n]$ is at most $H_{2n}(C_4)/2^{\lfloor(2n-2)/4\rfloor}$.

The study of $H_n(G)$ was initiated in \cite{KMS} where it was shown that the maximum number of Hamilton paths such that the union of any pair contains an odd cycle (of any length) is exactly $\binom{n}{\lfloor{n/2}\rfloor}$ for odd $n$ and $\frac{1}{2}\binom{n}{\lfloor{n/2}\rfloor}$ for even $n$. They asked if the same result holds when replacing an odd cycle of arbitrary length by a triangle. This question was answered positively by Kov\'acs and Solt\'esz \cite{KS}, determining $H_n(C_3)$ precisely. The authors generalized their work to other odd cycles in \cite{KS2}, giving bounds on $H_n(C_{2k+1})$. In particular they showed that $H_n(C_{2k+1}) = 2^{(1+o(1))n}$ when $k$ is a power of $2$, and they conjectured that the same result holds for all $k$. 

In the same paper, they showed that the maximum number of Hamilton paths such that the union of any pair contains an even cycle (of any length) is $\Omega\left(n!/n^2\right)$. After this, Cohen, Fachini, and K\"orner \cite{path_separation} showed that the behavior changes substantially if one replaces any even cycle with $C_4$, showing that behavior of the problem is different for even and odd cycles. In particular, they showed that 
\[
\left\lfloor \frac{n}{2}\right\rfloor ! \leq H_n(C_4) \leq n^{\frac{3}{4}n+o(n)},
\]
 leaving a superexponential gap in the upper and lower bounds. This work was built on by Solt\'esz \cite{S}, and finally in \cite{new_bounds} Harcos and Solt\'esz gave the current best-known upper bounds for $H_n(C_{2k})$, including determining the asymptotics of $\log H_n(C_4)$.

\begin{theorem}[Harcos and Solt\'esz \cite{new_bounds}] \label{new-bounds-result}
    We have
    $$n^{\frac{1}{2}n-\frac{1}{2}\frac{n}{\log  n}-O(1)}\le H_n(C_4)\le n^{\frac{1}{2}n+o\left(\frac{n}{\log n}\right)}.$$
    Moreover, for $k>2$ we have
    $$n^{\frac{1}{k}n-o(n)}\le H_n(C_{2k})\le n^{\left(1-\frac{1}{3k}\right)+o(n)}.$$
\end{theorem}

In this paper we will improve the upper bounds in \autoref{new-bounds-result} and also consider the cases $G=K_{2,3}$ and $G=K_{2,4}$. Our main results are the following.

\begin{theorem}
\label{mainresult}
We have
\begin{itemize}
    \item[(a)] $H_n(C_6)\le n^{\frac{2}{3}n+o(n)}$
    \item[(b)] $H_n(C_{8})\le n^{\frac{4}{5}n+o(n)}$
    \item[(c)] $H_n(C_{10})\le n^{\frac{4}{5}n+o(n)}.$
\end{itemize}
\end{theorem}

\begin{theorem}\label{lps-improvement}
    For every $k\ge 2$, we have
    $$H_n(C_{2k})\le n^{\left(1-\frac{2}{3k}\right)n+o(n)}.$$
\end{theorem}

\begin{theorem}\label{complete bipartite results} We have
$$H_n(K_{2,3})\le n^{\frac{1}{2}n}2^{-\frac{n}{2}+o(n)}$$
and
$$H_n(K_{2,4})\le n^{\frac{1}{2}n}3^{-\frac{n}{2}+o(n)}.$$
\end{theorem}

To obtain the upper bounds in \autoref{new-bounds-result} the authors of \cite{new_bounds} proved the following lemma which allows one to consider $G-$free graphs instead of directly attacking $H_n(C_{2k})$. Let $\ov{H_n(G)}$ be the maximum number of Hamilton paths in $K_n$ the union of any two of which does not contain $G$ subgraph. 

\begin{lemma}[Harcos and Solt\'esz \cite{new_bounds}] \label{gfreegraphs}
For every graph $G$ and every integer $n>1$, we have
$$H_n(G)\ov{H_n(G)}\le\frac{n!}{2}.$$
\end{lemma}

Thus, a lower bound on $\ov{H_n(G)}$ implies an upper bound on $H_n(G)$. The strategy used to obtain a lower bound on $\ov{H_n(G)}$ is to take a $G$-free graph and count its Hamilton paths (or cycles). To prove that a given $G-$free graph contains many Hamilton cycles, the authors of \cite{new_bounds} used the following result of Krivelevich.

\begin{theorem}[Krivelevich \cite{Krivelevich2012}] \label{krivelevich-result}
Suppose $G$ is a $d-$regular graph on $n$ vertices, all of whose nontrivial eigenvalues are at most $\lambda$ in absolute value. If the following conditions are satisfied:
\begin{itemize}
    \item $\frac{d}{\lambda}\ge(\log n)^{1+\varepsilon}$ for some constant $\varepsilon>0$;
    \item $\log d\cdot\log\frac{d}{\lambda}\gg\log n$,
\end{itemize}
then the number of Hamilton cycles in $G$ is $n!\left(\frac{d}{n}\right)^n(1+o(1))^n.$
\end{theorem}

This result can be interpreted as follows: the number of Hamilton cycles in a regular pseudorandom graph $G$ is asymptotically the same as the expected number of Hamilton cycles in the Erd\H{o}s-R\'enyi random graph $G(n,p)$, where the parameter $p$ is scaled to match the edge density of $G$ (and up to a very robust error term $(1+o(1))^n$).

We use the same technique of \cite{new_bounds} in which \autoref{gfreegraphs} allows us to focus on $C_{2k}-$free graphs. That is, using \autoref{gfreegraphs}, we can prove our upper bounds if we can construct $C_{2k}$-free graphs with enough Hamilton cycles. For \autoref{mainresult} we will use generalized quadrangles and generalized hexagons; for \autoref{lps-improvement} we will use the construction of Lubotzky, Phillips and Sarnak \cite{lubotzky1988}; and for \autoref{complete bipartite results} we will use the $K_{2,t}-$free graphs of F\"uredi \cite{FUREDI1996141}. These constructions are all irregular or bipartite, so we will prove versions of \autoref{krivelevich-result} which hold for irregular or bipartite graphs.

We note that for the generalized polygons and $K_{2,t}-$free graphs we use, there are regular bipartite and irregular non-bipartite versions of the construction which work equally well for the purposes of proving our main results.  However, if more precise bounds are desired then having the irregular non-bipartite version may make a difference. For example, using a polarity graph of some incidence structure instead of the bipartite incidence graph effectively doubles the density, so that the formula $n!\left(\frac{d}{n}\right)^n(1+o(1))^n$ increases by a factor of $2^n$. This factor becomes relevant when the exponent in the lower bound on $H_n(G)$ matches that of the upper bound, as is the case for $H_n(C_4)$ in \autoref{new-bounds-result}.

Below we give the irregular and bipartite versions of \autoref{krivelevich-result} which we will use.

\begin{theorem} Suppose $G$ is a graph on $n$ vertices with average degree $d$, with minimum and maximum degrees $\delta$ and $\Delta$, respectively, all of whose nontrivial eigenvalues are at most $\lambda$ in absolute value. If the following conditions are satisfied:
\label{hamiltonresult}
\begin{enumerate}
\item[(1)] $\Delta-\delta\le R\lambda$ for some constant $R>0$;
\item[(2)] $\frac{d}{\lambda}\gg\log^2n$; 
\item[(3)] $\log d\cdot\log\frac{d}{\lambda}\gg\log n$,
\end{enumerate}
Then the number of Hamilton cycles in $G$ is $n!\left(\frac{d}{n}\right)^n(1+o(1))^n$.
\end{theorem}

We make some remarks on the conditions (1)-(3). First, it will sometimes be convenient for us to assume that $\lambda = o(d)$ and condition (2) guarantees this. With this in mind, condition (1) then guarantees that we have $\frac{\Delta}{\delta} = 1+o(1)$. Next, if one uses $\lambda = \max\{\lambda_2, -\lambda_n\}$, then
using condition (1) may be undesirable when the graph is highly pseudo-random and $\lambda$ is small. Therefore we note that the condition can be replaced as follows with a condition that only references the degree parameters. Since $\Tr(A^k)$ is the number of closed walks of length $k$ in $A$ which start and end at the same vertex, we have $\Tr(A^2)\ge n\delta$. On the other hand, any eigenvalue of $G$ is at most $\Delta$, so 
$$\Tr(A^2)=\sum_i\lambda_i^2\le\Delta^2+(n-1)\lambda^2.$$
It follows that $\lambda^2\ge\frac{n\delta-\Delta^2}{n-1}$. If $\Delta=o(n)$ and $\frac{\Delta}{\delta}\to 1$, then the last quantity is $\Delta(1-o(1))$, implying $\lambda=\Omega(\sqrt{\Delta})$. Therefore, the conditions $\Delta=o(n)$ and $\Delta-\delta=O(\sqrt{\Delta})$ jointly imply (1). This will be strong enough for our purposes, however we make a final remark that if one needed to have a larger difference between maximum and minimum degree, then one could possibly choose $\lambda$ to be much larger than $\max\{\lambda_2, -\lambda_n\}$. Conditions (2) and (3) are quite mild, and in the frequently seen case that $d = n^a$ and $\max\{\lambda_2, -\lambda_n\} < d^{1-b}$, one may choose for example to take $\lambda$ to be any power of $d$ sufficiently close to $1$, which gives much more freedom for the degree sequence of $G$. One final remark is that the exponent $2$ in condition (2) is slightly weaker than the $1+\epsilon$ found in \cite{Krivelevich2012}. Because our graphs are not regular it is not as straightforward to obtain a lower bound on the permanent of our adjacency matrix, and this is where we use the stronger condition. We do not know if the $2$ can be replaced with $1+\epsilon$. 

We remark that \autoref{hamiltonresult} implies that the Hamilton cycles in a graph satisfying (1)-(3) are `evenly distributed' in the following sense. If (say) a bounded number of edges are deleted from each vertex, then using the Courant-Weyl inequalities one can show that (1)-(3) are still satisfied so the asymptotic formula $n!(d/n)^n(1+o(1))^n$ still holds. The allowance in \autoref{hamiltonresult} that the graph be irregular is required for this argument to work, even if the starting graph is regular. {We suspect that quantitative results of this type could be proved given more a more careful argument.}

\begin{theorem}\label{bipartite-version}
    Suppose $G$ is a $d-$regular bipartite graph on $n$ vertices with $\max\{\lambda_i:i\ne 1\}\le\lambda$, satisfying:
    \begin{itemize}
        \item[(5)] $\frac{d}{\lambda}\ge(\log n)^{1+\varepsilon}$ for some constant $\varepsilon>0$;
        \item[(6)] $\log d\cdot\log\frac{d}{\lambda}\gg\log n$.
    \end{itemize}
    Then the number of Hamilton cycles in $G$ is $n!\left(\frac{d}{n}\right)^n(1+o(1))^n$.
\end{theorem}

It is also possible to prove a version for irregular balanced bipartite graphs by combining the techniques used for \autoref{hamiltonresult} and \autoref{bipartite-version}.

Following \cite{Krivelevich2012}, we use the following notation and conventions for a given graph $G$: $n$ is the number of vertices; $\delta,d$ and $\Delta$ are the minimum, average, and maximum degrees respectively; $A$ is the adjacency matrix; and $\lambda_1\ge\cdots\ge\lambda_n$ are the adjacency eigenvalues, which satisfy $\max\{|\lambda_i|:i\ne 1\}\le\lambda$. {For a general symmetric $n\times n$ matrix $M$, we also use $\lambda_1(M)\ge\cdots\ge\lambda_n(M)$ to denote the eigenvalues of $M$}. A \textit{2-factor} in $G$ is a collection of vertex-disjoint cycles which covers $G$, where a single edge counts as a cycle. If $F$ is a 2-factor in $G$ then $c(F)$ is the number of cycles in $F$ of length at least $3$; for $s\in\NN$, $f(G,s)$ is the number of 2-factors in $G$ with exactly $s$ cycles; $f(G)$ is the total number of 2-factors in $G$; for an integer $2\le k\le n$, $\phi(G,k)$ is the maximum number of 2-factors in an induced subgraph of $G$ on $k$ vertices. For $U,W\subseteq V$, $e(U,W):=\{(u,w):u\in U,w\in W,u\sim w\}$; $e(U)$ is the number of edges with both endpoints in $U$; and $N(U)$ is the set of vertices not in $U$ which have a neighbor in $U$. We use the notation $f\ll g$ to mean $f = o(g)$.

This paper is organized as follows. In Section \ref{tools section} we collect several tools that we will need including an irregular expander mixing lemma, the lemmas we require to find Hamilton cycles, and results on permanents. In Section \ref{counting cycles section}, we prove \autoref{hamiltonresult} and \autoref{bipartite-version}. In Section \ref{main result section}, we prove our main results \autoref{mainresult}, \autoref{lps-improvement}, and \autoref{complete bipartite results}. {In Section \ref{lower lound section} we present a lower bound on $K_{2,3}$.}  Finally, in Section \ref{conclusion section}, we give concluding remarks and show that if a conjecture of Ustimenko holds, then the upper bound from Theorem \ref{new-bounds-result} can be improved for all $k>2$.

\section{Tools} \label{tools section}

\subsection{Expander mixing lemma}

Krivelevich and Sudakov in \cite{Krivelevich2006} describe how the expander mixing lemma can be modified in the irregular case without stating a quantitative result. We will repeat the key details of their proof sketch and show how the assumption (1) leads to an explicitly defined number controlling the edge distribution of $G$.

\begin{lemma}
\label{eml}
Let $G$ be a graph satisfying (1) and $\lambda \ll d$. Then for $n$ large enough, for all $U,W\subseteq V$ we have
$$\left|e(U,W)-\frac{d}{n}|U||W|\right|\le\ov\lambda\sqrt{|U||W|},$$
where $\ov\lambda=(10R+1)\lambda$.
\end{lemma}
\begin{proof} Let $x_1,\ldots,x_n$ be a basis of $\RR^n$ consisting of orthonormal eigenvectors of $A$. Then we have $A=A_1+\s E$, where $A_1=\lambda_1x_1x_1^t$ and $\s E=\sum_{i=2}^n\lambda_ix_ix_i^t$. Let $u=|U|,w=|W|$ and denote by $\chi_U,\chi_W$ the characteristic vectors of $U$ and $W$, respectively. Represent these vectors in the basis $x_1,\ldots,x_n$ by $\chi_U=\sum_{i=1}^n\alpha_ix_i$, $\chi_W=\sum_{i=1}^n\beta_ix_i$. Using the same arguments as for the regular expander mixing lemma we obtain
$$e(U,W)=\chi_U^tA_1\chi_W+\chi_U^t\s E\chi_W,$$
where
$$\chi_U^tA_1\chi_W=\alpha_1\beta_1\lambda_1 \quad \mbox{ and } \quad |\chi_U^t\s E\chi_W|\le\lambda\sqrt{uw}.$$
Now by expanding the vector $\frac{1}{\sqrt n}(1,\ldots,1)$ in the basis $x_1,\ldots,x_n$ we find (see \cite{Krivelevich2006}):
$$\begin{aligned}
&\alpha_1=\frac{u}{\sqrt n}+\varepsilon_1,\ |\varepsilon_1|\le\frac{\sqrt{\frac{2Ku}{n}}}{d-\lambda},\\
&\beta_1=\frac{w}{\sqrt n}+\varepsilon_2,\ |\varepsilon_2|\le\frac{\sqrt{\frac{2Kw}{n}}}{d-\lambda},\\
&\lambda_1=d+\varepsilon_3,\ |\varepsilon_3|\le\frac{n(d-\lambda)^2}{n(d-\lambda)^2-K}\sqrt{\frac{K}{n}},\end{aligned}$$
where $K:=\sum_{v\in V}(d(v)-d)^2$.
Therefore,
$$\begin{aligned}
\alpha_1\beta_1\lambda_1&=\frac{duw}{n}+\varepsilon_1\frac{dw}{\sqrt n}+\varepsilon_2\frac{du}{\sqrt n}+\varepsilon_3\frac{uw}{n}\\
&\ \ +\varepsilon_1\varepsilon_2 d+\varepsilon_1\varepsilon_3\frac{w}{\sqrt n}+\varepsilon_2\varepsilon_3\frac{u}{\sqrt n}\\
&\ \ +\varepsilon_1\varepsilon_2\varepsilon_3.\end{aligned}$$
Note that $\frac{K}{n}\le(\Delta-\delta)^2\le R^2\lambda^2$. Since we have $d\gg \lambda$ and $u,w\leq n$ we will use for example that 
\[
\frac{d}{d-\lambda}, \frac{(d-\lambda)^2}{(d-\lambda)^2 - \frac{K}{n}} \to 1, \quad \frac{u}{\sqrt{n}}, \frac{w}{\sqrt{n}} \leq \sqrt{uw}.
\]
We estimate each error term above: 
$$\begin{aligned}
\left|\varepsilon_1\frac{dw}{\sqrt n}\right|&\le\sqrt 2\sqrt{\frac{K}{n}}\frac{d}{d-\lambda}\sqrt{uw}\le 2R\lambda\sqrt{uw}\\
\left|\varepsilon_2\frac{du}{\sqrt n}\right|&\le\sqrt 2\sqrt{\frac{K}{n}}\frac{d}{d-\lambda}\sqrt{uw}\le2R\lambda\sqrt{uw}\\
\left|\varepsilon_3\frac{uw}{n}\right|&\le\frac{(d-\lambda)^2}{(d-\lambda)^2-K/n}\sqrt{\frac{K}{n}}\frac{u}{\sqrt n}\frac{w}{\sqrt n}\le 2R\lambda\sqrt{uw}\\
\left|\varepsilon_1\varepsilon_2d\right|&\le 2\frac{K}{n}\frac{d}{(d-\lambda)^2}\sqrt{uw} \le R\lambda\sqrt{uw}\\
\left|\varepsilon_1\varepsilon_3\frac{w}{\sqrt n}\right|&\le\sqrt 2\frac{\sqrt{K/n}}{d-\lambda}\frac{(d-\lambda)^2}{(d-\lambda)^2-K/n}\sqrt{\frac{K}{n}}\sqrt{u}\frac{w}{\sqrt n}\le 2\frac{K/n}{d-\lambda}\sqrt{uw}\leq R\lambda\sqrt{uw}\\
\left|\varepsilon_2\varepsilon_3\frac{u}{\sqrt n}\right|&\le \sqrt 2\frac{\sqrt {K/n}}{d-\lambda}\frac{(d-\lambda)^2}{(d-\lambda)^2-K/n}\sqrt{\frac{K}{n}}\sqrt w\frac{u}{\sqrt n}\le 2\frac{K/n}{d-\lambda}\sqrt{uw}\leq R\lambda\sqrt{uw}\\
\left|\varepsilon_1\varepsilon_2\varepsilon_3\right|&\le 2\frac{K/n}{(d-\lambda)^2}\frac{(d-\lambda)^2}{(d-\lambda)^2-K/n}\sqrt{\frac{K}{n}}\sqrt{uw}\le R\lambda\sqrt{uw}.\end{aligned}$$
We combine these inequalities to conclude:
$$\left|e(U,W)-d\frac{uw}{n}\right|\le(1+3\cdot 2R+4\cdot 1R)\lambda\sqrt{uw}.$$
\end{proof}

A bipartite version of the expander-mixing lemma can be seen at least as early as Haemers's thesis, in the language of design theory \cite{haemers}. We will use the following version, which is stated and proved for example in Lemma 8 of \cite{de-winter2012}.
\begin{lemma}\label{bipartite-eml}
If $G$ is a $d-$regular bipartite graph with parts $X$ and $Y$ with second eigenvalue at most $\lambda$, then for every $U\subseteq X,W\subseteq Y$ we have
$$\left|e(U,W)-\frac{2d|U||W|}{n}\right|\le\lambda\sqrt{|U||W|\left(1-\frac{|U|}{n}\right)\left(1-\frac{|W|}{n}\right)}.$$
\end{lemma}

\subsection{Constructions}

We now describe the graphs used in our application and show that these graphs satisfy the assumptions of \autoref{hamiltonresult} or \autoref{bipartite-version}.

\subsubsection{Generalized polygons}

For background on these graphs, see p. 507 of \cite{LAZEBNIK1999503}. We begin with the generalized quadrangles. If $q$ is a prime power and $n=2(q^3+q^2+q+1)$ let $X_n$ denote the generalized quadrangle of type $B_2(q)$. This is a $(q+1)-$regular bipartite graph on $n$ vertices which contains no $C_6$. The eigenvalues of $NN^T$, where $N$ is an incidence matrix of $X_n$ are given in \cite{tanner}: they are $(q+1)^2$, $2q$, and $0$. Therefore the adjacency eigenvalues of $X_n$ satisfy $\lambda_2\le\sqrt{2q}$ so that $\frac{d}{\lambda_2}=\Theta(n^{1/6})$ and conditions (5) and (6) of \autoref{bipartite-version} are satisfied.

The situation is similar for the generalized hexagons. If $q$ is a prime power and $n=2(q^5+q^4+q^3+q^2+q+1)$, let $Y_n$ be the generalized hexagon of type $G_2(q)$. This is a $(q+1)-$regular bipartite graph on $n$ vertices which contains no $C_{10}$. The eigenvalues of $NN^T$ are $(q+1)^2,3q+3,2q+1$. Therefore the adjacency eigenvalues of $Y_n$ satisfy $\lambda_2\le \sqrt{3q+3}$ so that $\frac{d}{\lambda_2}=\Theta(n^{1/10})$ and conditions (5) and (6) of \autoref{bipartite-version} are satisfied.

\subsubsection{Lubotzky-Phillips-Sarnak graphs}

The upper bound on $H_n(C_{2k})$ for general $k$ obtained in \cite{new_bounds} used the non-bipartite LPS graphs \cite{lubotzky1988}. With \autoref{bipartite-version} we can use the bipartite LPS graphs, which have higher girth compared to their density. If $p\ne q$ are primes congruent to $1\pmod 4$ with $\left(\frac{p}{q}\right)=-1$, then there is a bipartite graph $X^{p,q}$ on $n=q(q^2-1)$ vertices which is $(p+1)-$regular, whose girth is $g(X^{p,q})\ge 4\log_pq-\log_p4$, and whose adjacency eigenvalues satisfy $\lambda_2\le 2\sqrt p$. We will use (see Section \ref{main result section}) an infinite sequence of pairs $(p,q)$ for which $q\in(p^{k/2+\delta},(1+\varepsilon)p^{k/2+\delta})$, where $\delta,\varepsilon$ are constants, in order to have $C_{2k}-$free graphs. For such $(p,q)$ we have that $\frac{d}{\lambda_2}=\Theta(n^{k/12+\delta/6})$ and so (5) and (6) in \autoref{bipartite-version} are satisfied.

\subsubsection{F\"uredi's graphs}

Let $t\in\NN$ and let $q$ be a prime power such that $t|q-1$ (we are interested in the cases $t=2$ and $t=3$). In \cite{FUREDI1996141} F\"uredi constructed $K_{2,t+1}-$free graphs on $n=(q^2-1)/t$ vertices in which every vertex has degree $q$ or $q-1$; we will denote these graphs by $Z_n$ if $t=2$ and $W_n$ in $t=3$. In \cite{FUREDI1996141} the following property of these graphs is shown: for each vertex $x$, there are $(q-1)/t-1$ vertices which have no common neighbors with $x$, and every other $y\ne x$ has exactly $t$ common neighbors with $x$. Thus, in each row of the matrix $A^2$ there is either $q$ or $q-1$ in the diagonal entry, $(q-1)/t-1$ 0's, and $(q^2-1)/t-(q-1)/t$ $t$'s, and we can write
$$A^2=tJ+(q-t)I+E$$
where $J$ is the all-ones matrix and $E$ is a symmetric matrix with $0$ or $-1$ on the diagonal, $(q-1)/t-1$ entries $-t$ in each row and $0$ elsewhere. The absolute value of the eigenvalues of $E$ are bounded by the maximum absolute value of a row sum of $E$, {which is $0+t((q-1)/t-1)=q-t-1$. The eigenvalues of $tJ+(q-t)I$ are $t(q^2-1)/t+q-t=q^2+q-1-t$ (with multiplicity 1) and $q-t$. Applying the Courant-Weyl inequalities (see e.g. \cite{horn_johnson_1985}, p. 367):
$$-q+t+1\le\lambda_n(E)\le \lambda_i(A^2)-\lambda_i(tJ+(q-t)I)\le \lambda_1(E)\le q-t-1$$
so $\lambda_2(A^2)\le q-t+q-t-1=2q-2t-1$. Since the first eigenvalue of $A$ is at least $q-1$ and $q-1>\sqrt{2q-2t-1}$, it follows that any other eigenvalue of $A$ is at most $\sqrt{2q-2t-1}$. Thus $\frac{d}{\lambda}=\Theta(n^{1/4})$ and $Z_n$, $W_n$ each satisfy (1), (2), (3) of \autoref{hamiltonresult}.}

\subsection{Finding Hamilton cycles}

When applying \autoref{eml} in the proof of \autoref{hamiltonresult} it is most convenient to avoid dealing with the relationship between $\lambda$ and $\ov\lambda$.

\label{paramterremark}
\begin{remark} Conditions (1)-(3) each imply the corresponding condition below:
\begin{itemize}
    \item[(1')] $\Delta-\delta\le\ov\lambda$, where $\ov\lambda$ is a number controlling the edge distribution of $G$ as in \autoref{eml};
    \item[(2')] $\frac{d}{\ov\lambda}\gg\log^2n$; 
    \item[(3')] $\log d\cdot\log\frac{d}{\ov\lambda}\gg\log n$.
\end{itemize}
\end{remark}

Krivelevich and Sudakov showed in \cite{Krivelevich2003} that regular pseudorandom graphs (satisfying conditions weaker than \autoref{krivelevich-result}) contain a Hamilton cycle, and Alexander \cite{alexander2016} extended their result to bipartite graphs. In \cite{Krivelevich2003} a series of corollaries of the expander mixing lemma are used throughout the proof of the main result. We can prove analogues of these for the irregular and bipartite cases using relatively minor modifications to their arguments. Using these corollaries we obtain the following analogues of Lemma 3.1 of \cite{Krivelevich2012}; the proofs are also in the appendix.

\begin{restatable}{lemma}{lemmatwo}
\label{pathlemma} For a graph $G$ satisfying conditions (1')-(3'), there exists $C>0$ and $n_0>0$ such that for every $n\ge n_0$ the following is true. Let $P$ be a path in $G$. Then there is a path $P^*$ in $G$ from vertices $a$ to $b$ such that:
\begin{itemize}
    \item[1.] $V(P^*)=V(P)$;
    \item[2.] $|E(P)\Delta E(P^*)|\le\frac{C\log n}{\log\frac{d}{\ov\lambda}}$;
    \item[3.] $ab\in E(G)$ or $G$ contains an edge between $\{a,b\}$ and $V-V(P^*)$.
\end{itemize}
\end{restatable}

\begin{restatable}{lemma}{bipartitelemmatwo}
\label{bipartite-path}
For a graph $G$ satisfying conditions (5) and (6), there exists $C>0$ and $n_0>0$ such that for every $n\ge n_0$ the following is true. Let $P$ be an odd-length path in $G$. Then there is a path $P^*$ in $G$ from vertices $a$ to $b$ such that:
\begin{enumerate}
    \item $V(P^*)=V(P)$;
    \item $|E(P^*)\Delta E(P)|\le\frac{C\log n}{\log\frac{d}{\lambda}}$;
    \item $ab\in E(G)$ or $G$ contains an edge between $\{a,b\}$ and $V-V(P^*)$.
\end{enumerate}
\end{restatable}

\subsection{Permanent estimates}

The following upper bound on the permanents was stated in \cite{Krivelevich2012}; it is a consequence of a conjecture of Minc which was proved by Br\'egman in \cite{bregman}.

\begin{lemma} 
\label{bregmantheorem}
Let $A$ be an $n\times n$ $\{0,1\}-$matrix with $t$ ones. Then $per(A)\le\prod_{i=1}^n(r_i!)^{1/r_i}$, where $r_i$ are integers satisfying $\sum_{i=1}^nr_i=t$ and as equal as possible.
\end{lemma}

For the lower bound we will use the van der Waerden conjecture, which was proved by Egorychev \cite{egorychev} and Falikman \cite{falikman}.

\begin{lemma}
\label{vdw}
Let $A$ be an $n\times n$ doubly stochastic matrix. Then $per(A)\ge\frac{n!}{n^n}$.
\end{lemma}

In the bipartite case, we can apply \autoref{vdw} directly, but in the irregular case we cannot because if $G$ is irregular, then $\frac{1}{d}A$ is not doubly stochastic. Thus, in order to obtain a lower bound on $per(A)$ we want to know conditions under which it is possible to decrease the entries of $A$ to obtain a doubly stochastic matrix. When this is possible, $A$ is called \textit{doubly superstochastic}. The following equivalence is a consequence of Theorem 4 of \cite{MIRSKY196830}. 

\begin{lemma} \label{dsscondition}
An $n\times n$ matrix $A=[a_{ij}]$ with nonnegative entries is doubly superstochastic if and only if, for all $I,J\subseteq\{1,\ldots,n\}$, we have
\begin{equation} \label{dssequation}
\sum_{i\in I,j\in J}a_{ij}\ge|I|+|J|-n.
\end{equation}
\end{lemma}

\begin{corollary}
\label{doubly-superstochastic}
If $A$ is the adjacency matrix of a graph $G$ satisfying (1')-(3'), then for $n$ large enough, $\frac{1}{\delta-9\ov\lambda}A$ is doubly superstochastic.
\end{corollary}
\begin{proof}
Let $I,J\subseteq [n]$ and denote $x:=|I|,y:=|J|$. Note that when we are considering the matrix $\frac{1}{\delta-9\ov\lambda}A$ in \autoref{doubly-superstochastic}, the sum in equation \eqref{dssequation} is equal to $\frac{1}{\delta-9\ov\lambda}e(I,J)$.
We may assume without loss of generality that $y\le x$. Moreover if $y=0$ or $x=n$ or $x\le\frac{n}{2}$ then equation \eqref{dssequation} is immediate, so we assume that $1\le y\le x$ and $\frac{n}{2}\le x\le n-1$.

\textit{Case 1:} $y\le\frac{n}{2}\le x$. Then by Lemma \ref{eml}
$$\begin{aligned}
\frac{1}{\delta-9\ov\lambda}e(I,J)&=\frac{1}{\delta-9\ov\lambda}\left(e(V,J)-e(V-I,J)\right)\\
&\ge\frac{\delta}{\delta-9\ov\lambda}y-\frac{d}{\delta-9\ov\lambda}\frac{y(n-x)}{n}-\frac{\ov\lambda}{\delta-9\ov\lambda}\sqrt{y(n-x)}.
\end{aligned}
$$
Set $z:=n-x$, so $1\le z\le\frac{n}{2}$. Moreover if $x+y<n$ we are done, so we may assume $y\ge z$. We must show the following function is nonnegative on the region $\{1\le z\le y\le n/2\}$:
$$\begin{aligned}
g(y,z):=&\frac{\delta}{\delta-9\ov\lambda}y-\frac{d}{\delta-9\ov\lambda}\frac{yz}{n}-\frac{\ov\lambda}{\delta-9\ov\lambda}\sqrt{yz}-y+z\\
=&\frac{9\ov\lambda}{\delta-9\ov\lambda}y-\frac{d}{(\delta-9\ov\lambda)n}yz-\frac{\ov\lambda}{\delta-9\ov\lambda}\sqrt{yz}+z.\end{aligned}
$$
Taking partial derivatives and multiplying by $y$ and $z$, respectively, we find that at any critical point of $g$ we must have
$$\begin{aligned}
0&=\frac{9\ov\lambda}{\delta-9\ov\lambda}y-\frac{d}{(\delta-9\ov\lambda)n}yz-\frac{\ov\lambda}{2(\delta-9\ov\lambda)}\sqrt{yz}\\
0&=z-\frac{d}{(\delta-9\ov\lambda)n}yz-\frac{\ov\lambda}{2(\delta-9\ov\lambda)}\sqrt{yz}
\end{aligned}
$$
which implies
$$y=\frac{\delta-9\ov\lambda}{9\ov\lambda}z.$$
However, the function 
$$g\left(\frac{\delta-9\ov\lambda}{9\ov\lambda}z,z\right)=z-\frac{d}{9\ov\lambda n}z^2-\frac{\sqrt{\ov \lambda}}{\sqrt {9(\delta-9\ov\lambda)}}z+z$$
has no local minimum. Therefore the minimum of $g$ occurs on the boundary of the region. We have: 
$$\begin{aligned}
g(y,1)&=\frac{9\ov\lambda}{\delta-9\ov\lambda}y-\frac{d}{(\delta-9\ov\lambda)n}y-\frac{\ov\lambda}{\delta-9\ov\lambda}\sqrt y+1\\
&\ge\sqrt{y}\left(\left(\frac{9\ov\lambda}{\delta-9\ov\lambda}-\frac{d}{(\delta-9\ov\lambda)n}\right)\sqrt y-\frac{\ov\lambda}{\delta-9\ov\lambda}\right)\\
&\ge\sqrt y\left(\frac{9\ov\lambda}{\delta-9\ov\lambda}-\frac{d}{(\delta-9\ov\lambda)n}-\frac{\ov\lambda}{\delta-9\ov\lambda}\right)\ge 0;\\
g\left(\frac{n}{2},z\right)&=\frac{9\ov\lambda n}{2(\delta-9\ov\lambda)}-\frac{d}{2(\delta-9\ov\lambda)}z-\frac{\ov\lambda\sqrt n}{\sqrt 2(\delta-9\ov\lambda)}\sqrt z+z\\
&\ge\left(\frac{9\ov\lambda }{2(\delta-9\ov\lambda)}-\frac{\ov\lambda }{\sqrt 2(\delta - 9\ov\lambda)}\right)n+\left(1-\frac{d}{2(\delta-9\ov\lambda)}\right)z\ge0+0;\\
g(y,y)&=\left(\frac{9\ov\lambda}{\delta-9\ov\lambda}-\frac{\ov\lambda}{\delta-9\ov\lambda}+1\right)y-\frac{d}{(\delta-9\ov\lambda)n}y^2.
\end{aligned}
$$
Since the last function has no local minimum, its minimum value occurs at a vertex of the region, and must therefore be covered by one of the previous two cases $(y,1)$ or $(n/2,2)$.

\textit{Case 2:} $\frac{n}{2}\le y\le x$. Then by Lemma \ref{eml}
$$\begin{aligned}
\frac{1}{\delta-9\ov\lambda}e(I,J)&=\frac{1}{\delta-9\ov\lambda}\left(e(V,V)-e(V-I,V)-e(I,V-J)\right)\\
&\ge\frac{dn}{\delta-9\ov\lambda}-\frac{d(n-x)n}{(\delta-9\ov\lambda)n}-\frac{dx(n-y)}{(\delta-9\ov\lambda)n}-\frac{\ov\lambda}{\delta-9\ov\lambda}\sqrt{(n-x)n}-\frac{\ov\lambda}{\delta-9\ov\lambda}\sqrt{x(n-y)}\\
&\ge\frac{dxy}{(\delta-9\ov\lambda)n}-\frac{2\ov\lambda n}{\delta-9\ov\lambda}.\\
\end{aligned}$$
Therefore, 
$$\begin{aligned}
\frac{1}{\delta-9\ov\lambda}e(I,J)-(x+y-n)&\ge\left(\frac{d}{(\delta-9\ov\lambda)}-1\right)\frac{xy}{n}+\frac{xy}{n}-x-y+n-\frac{2\ov\lambda n}{\delta-9\ov\lambda}\\
&\ge\frac{9\ov\lambda}{\delta-9\ov\lambda}\frac{xy}{n}+\frac{1}{n}(n-x)(n-y)-\frac{2\ov\lambda n}{\delta-9\ov\lambda}\\
&\ge\frac{9\ov\lambda n}{4(\delta-9\ov\lambda)}-\frac{2\ov\lambda n}{\delta-9\ov\lambda}\ge 0.
\end{aligned}
$$
\end{proof}
From (2$'$), \autoref{vdw} and \autoref{doubly-superstochastic}, we obtain that for such graphs we have
\begin{equation} \label{permanent-estimate}
\begin{aligned}
per(A)&\ge\left(\frac{\delta-9\ov\lambda}{e}\right)^n\ge\left(\frac{\Delta-10\ov\lambda}{e}\right)^n=\left(\frac{\Delta}{e}\right)^n\left(1-\frac{10\ov\lambda}{\Delta}\right)^n\\
&\ge\left(\frac{\Delta}{e}\right)^n\left(\frac{1}{2e}\right)^{\frac{10\ov\lambda n}{\Delta}}\ge\left(\frac{\Delta}{e}\right)^ne^{\frac{-10(1+\log 2)n}{\log^2 n}}\ge\left(\frac{\Delta}{e}\right)^ne^{-\frac{20n}{\log^2\Delta}}
\end{aligned}
\end{equation}

\section{Hamilton cycles} \label{counting cycles section}

We now prove \autoref{hamiltonresult} and \autoref{bipartite-version}. Our argument follows \cite{Krivelevich2012}, but we must take into account the fact that the graph is irregular or bipartite. For the upper and lower bounds, we utilize the connection between the permanent of $A$ and $2$-factors in the graph. The upper bound follows quickly using \autoref{bregmantheorem}. A high-level outline of the proof of the lower bound is as follows:

\begin{itemize}
    \item Using \autoref{vdw} and the connection between the permanent of $A$ and $2$-factors in the graph, we have a lower bound on a weighted count of the $2$-factors in $G$.
    \item We show that the number of $2$-factors with ``many" distinct cycles is negligible compared to the total.
    \item For each $2$-factor with ``few" distinct cycles, \autoref{pathlemma} or \autoref{bipartite-path} is used to convert it to a Hamilton cycle with few rotations.
    \item For any given Hamilton cycle, the number of ways that a $2$-factor can be converted to it by few rotations is bounded above. 
    \item Combining all of the bounds, we obtain a lower bound on the number of distinct Hamilton cycles in the graph.
\end{itemize}

We now proceed with the details.

\subsection{Irregular case}

First note that $per(A)$ counts the number of oriented 2-factors of $G$ (where an orientation is applied to each cycle), so $per(A)=\sum 2^{c(F)}$ with the sum taken over all $2-$factors $F$. To get an upper bound on the number of Hamilton cycles, we apply \autoref{bregmantheorem} to the matrix $A$: 
$$\begin{aligned}
h(G)&\le per(A)\le(\lceil d\rceil !)^{\frac{n}{\left\lfloor d\right\rfloor}}\le\left(\lceil d\rceil(\lceil d\rceil -1)\left(\frac{\lceil d\rceil -1}{e}\right)^{\lceil d\rceil -1}\right)^\frac{n}{d-1}\le\left((d+1)^2\left(\frac{d}{e}\right)^d\right)^\frac{n}{d-1}\\
&=\left(\frac{d}{e}\right)^n\left[\left(\frac{d}{e}\right)^\frac{n}{d-1}(d+1)^\frac{2n}{d-1}\right]=\left(\frac{d}{e}\right)^n(1+o(1))^n.
\end{aligned}$$

We now focus on the lower bound. Applying equation \eqref{permanent-estimate}, we have
\begin{equation} \label{permanentsum}
\left(\frac{\Delta}{e}\right)^ne^{-\frac{20n}{\log^2\Delta}}\le per(A)=\sum_{s=1}^{n/2}f(G,s)2^s.
\end{equation}

Moving to the second step in the proof, we show that the $2$-factors with too many cycles contribute negligibly to the sum. Set $s^*:=20n/\log^2\Delta$. Let $s>s^*$. Define $s_1:=4s/\log\Delta$. If a 2-factor $F$ has $s$ cycles then the $s_1$ shortest cycles have total length $t\le 4n/\log\Delta$. Fix $t\le\frac{4n}{\log\Delta}$, and let $(k_1^{(m_1)},\ldots,k_\ell^{(m_\ell)})$ be a multiset representing the $s_1$ shortest cycle lengths, listed in increasing order. Thus $m_1+\cdots+m_\ell=s_1$ and $k_1m_1+\cdots+k_\ell m_\ell=t$. We want to estimate the number $N\left(k_1^{(m_1)},\ldots,k_\ell^{(m_\ell)}\right)$ of $2$-factors whose $s_1$ shortest cycles give this multiset. To count this, first choose a vector of disjoint sets of vertices $(W_1\ldots,W_\ell)$ with $|W_i|=m_i$. Then for each $i$, for each $v\in W_i$, choose a path of length $k_i-1$ starting at $v$, and close the path to a cycle if possible. Finally cover the remaining $n-t$ vertices with a 2-factor. Each 2-factor with $s_1$ shortest cycle lengths $(k_1^{(m_1)},\cdots,k_\ell^{(m_\ell)})$ occurs at least $k_1^{m_1}\cdots k_\ell^{m_\ell}$ times. On the other hand, there are at most ${n\choose m_1}\cdots{n\choose m_\ell}\Delta^{m_1(k_1-1)}\cdots\Delta^{m_\ell(k_\ell-1)}\phi(G,n-t)$ ways to execute this process, so we obtain
$$\begin{aligned}
N\left(k_1^{(m_1)},\ldots,k_\ell^{(m_\ell)}\right)&\le\frac{{n\choose m_1}\cdots{n\choose m_\ell}}{k_1^{m_1}\cdots k_\ell^{m_\ell}}\cdot\Delta^{t-s_1}\cdot\phi(G,n-t)\\
&\le\frac{(en)^{s_1}}{(k_1m_1)^{m_1}\cdots(k_\ell m_\ell)^{m_\ell}}\cdot\Delta^{t-s_1}\cdot\phi(G,n-t).
\end{aligned}$$
We want to estimate the denominator above. Along the lines of Lemma 11.4 in \cite{van_lint_wilson_2001}, we use the fact that the function $x\log x$ on $(0,\infty)$ is convex. Let $\gamma=\frac{1}{\frac{1}{k_1}+\cdots+\frac{1}{k_\ell}}$, and note that since $\frac{1}{k_1}+\cdots+\frac{1}{k_\ell}\le\frac{1}{2}+\frac{1}{3}+\cdots+\frac{1}{s_1}\le\log(s_1)$, we have $\gamma\ge\frac{1}{\log(s_1)}$. Now we have
\[
\left(\sum_{i=1}^\ell\frac{\gamma}{k_i}m_ik_i\right)\log\left(\sum_{i=1}^\ell\frac{\gamma}{k_i}m_ik_i\right)\le\sum_{i=1}^\ell\frac{\gamma}{k_i}m_ik_i\log(m_ik_i).
\]
Simplifying gives
\[
s_1\log(\gamma s_1)\le\sum_{i=1}^\ell\log\left((m_ik_i)^{m_i})\right),\]
and finally using the lower bound on $\gamma$, we have
\[
\left(\frac{s_1}{\log s_1}\right)^{s_1}\le(\gamma s_1)^{s_1}\le(m_1k_1)^{m_1}\cdots(m_\ell k_\ell)^{m_\ell}.
\]
This implies that
\begin{equation} \label{vectorfix}
N\left(k_1^{(m_1)},\ldots,k_\ell^{(m_\ell)}\right)\le\left(\frac{en\log s_1}{s_1}\right)^{s_1}\cdot\Delta^{t-s_1}\cdot\phi(G,n-t).
\end{equation}

Now we estimate the term $f(G,s)2^s$ of equation \eqref{permanentsum}. To choose a 2-factor with $s$ cycles, we first choose a length $2s_1\le t\le\frac{4n}{\log\Delta}$. Then we choose a vector $(k_1,\ldots,k_{s_1})$ of cycle lengths such that $k_1+\cdots+k_{s_1}=t$; the number of possibilities is at most the number of solutions of $k_1+\cdots+k_{s_1}=t$ in positive integers, which is at most ${t+s_1\choose s_1}\le\left(\frac{6t}{s_1}\right)^{s_1}$. Finally we choose a 2-factor whose $s_1$ shortest cycles have lengths $k_1,\ldots,k_{s_1}$; we estimate the number of possibilities using equation \eqref{vectorfix}. This gives:
\begin{equation} \label{bigsum}
    f(G,s)\cdot 2^s\le\sum_{t\le \frac{4n}{\log\Delta}}\left(\frac{en\log s_1}{\Delta s_1}\right)^{s_1}\left(\frac{6t}{s_1}\right)^{s_1}\cdot\Delta^t\cdot\phi(G,n-t)\cdot 2^s.
\end{equation}
To estimate $\phi(G,n-t)$, we find a uniform upper bound on $per(A_1)$, where $A_1$ is the adjacency matrix of an induced subgraph of $n-t$ vertices. Let $V_0\subseteq V$ with $|V_0|=t$. Let $A_1$ be the adjacency matrix of $G[V-V_0]$. The number of 2-factors of $G[V-V_0]$ is at most $per(A_1)$. From \autoref{emlcor1} we have $e_G(V_0)\le\frac{t^2}{2}\frac{\Delta}{n}+\ov\lambda t$. It thus follows that
$$\begin{aligned}
e_G(V_0,V-V_0)\ge\delta t-\frac{\Delta t^2}{n}-2\ov\lambda t\ge \Delta t-\frac{\Delta t^2}{n}-3\ov\lambda t.\\
\end{aligned}$$
We derive
$$\begin{aligned}
2e_G(V-V_0)\le\Delta(n-t)-\Delta t+\frac{\Delta t^2}{n}+3\ov\lambda t.\end{aligned}$$
Thus the average degree in $G[V-V_0]$ is
$$\begin{aligned}
\frac{2e_G(V-V_0)}{n-t}&\le \Delta-\frac{\Delta t}{n-t}+\frac{\Delta t^2}{n(n-t)}+\frac{3\ov\lambda t}{n-t}\\
&=\Delta \left(1-\frac{t}{n}\right)+\frac{3 \ov\lambda t}{n-t}=:d_1.\end{aligned}$$
Then by \autoref{bregmantheorem},
$$\begin{aligned}
per(A_1)&\le\left(\lceil d_1\rceil!\right)^{\frac{n-t}{\lfloor d_1\rfloor}}\le\left(\left(\frac{d_1}{e}\right)^{d_1}\cdot d_1(d_1+1)\right)^{\frac{n-t}{\lfloor d_1\rfloor}}\le \left(\frac{d_1}{e}\right)^{n-t}\cdot\left(\frac{d_1}{e}\right)^{n\left(\frac{d_1}{\lfloor d_1\rfloor}-1\right)}\cdot d_1^{\frac{2n}{\lfloor d_1\rfloor}}\cdot 2^{\frac{n}{\lfloor d_1\rfloor}}\\
&\le\left(\frac{d_1}{e}\right)^{n-t}\cdot d_1^{\frac{2n}{d_1}}\cdot d_1^{\frac{2n}{d_1}}\cdot 2^{\frac{n}{d_1}}\le\left(\frac{d_1}{e}\right)^{n-t}e^{\frac{5n\log \Delta}{\Delta}}.\end{aligned}$$

We the substitute $d_1$ into the expression above; using
$$\Delta\left(1-\frac{t}{n}\right)+\frac{3\ov\lambda t}{n-t}\le\Delta\left(1-\frac{t(n-t)}{n(n-t)}\right)\left(1+\frac{4\ov\lambda t}{\Delta(n-t)}\right)$$
we obtain
$$per(A_1)\le\frac{\Delta^{n-t}}{e^n}\exp\left\{\frac{t^2}{n}+\frac{4\ov\lambda t}{\Delta}+\frac{5n\log\Delta}{\Delta}\right\}.$$
This is an upper bound on $\phi(G,n-t)$. Plugging this into equation \eqref{bigsum} we obtain:
\begin{equation}\label{fgs-sum}
f(G,s)\cdot 2^s
\le\left(\frac{\Delta}{e}\right)^n\sum_{t\le\frac{4n}{\log\Delta}}\left(\frac{en\log s_1}{\Delta s_1}\right)^{s_1}\left(\frac{6t}{s_1}\right)^{s_1}\cdot 2^s\cdot\exp\left\{\frac{t^2}{n}+\frac{4\ov\lambda t}{\Delta}+\frac{5n\log\Delta}{\Delta}\right\}.
\end{equation}
Now $e\cdot 6\cdot2^{(\log\Delta)/4}<18\cdot2^{(\log\Delta)/4}=18\Delta^{(\log 2)/4}\le\Delta^{1/4}$. Recalling that $s = s_1\log \Delta/4$, we have that the $t^{th}$ term in the sum above is at most
$$\left(\frac{18nt\log s_1}{\Delta s_1^2}\right)^{s_1}\cdot 2^{\frac{s_1\log\Delta}{4}}\cdot\exp\left\{\frac{t^2}{n}+\frac{4\ov\lambda t}{\Delta}+\frac{5n\log\Delta}{\Delta}\right\}\le\left(\frac{nt\log s_1}{\Delta^{3/4}s_1^2}\right)^{s_1}\exp\left\{\frac{t^2}{n}+\frac{4\ov\lambda t}{\Delta}+\frac{5n\log\Delta}{\Delta}\right\}.$$
Since $s_1\ge\frac{80n}{\log^3\Delta}$, using (3$'$) we have $s_1^2\ge\frac{n^2}{\log^6\Delta}\ge\frac{4n^2\log n}{\Delta^{0.05}\log\Delta}\ge\frac{nt\log s_1}{\Delta^{0.05}}$, so 
$$\begin{aligned}
\left(\frac{\Delta^{3/4}s_1^2}{nt\log s_1}\right)^{s_1}&\ge\left(\frac{\Delta^{3/4}}{\Delta^{0.05}}\right)^{s_1}=\Delta^{0.7s_1}\ge e^{\frac{50 n}{\log^2\Delta}}.\end{aligned}$$
For the terms in the exponent $\exp\{\ldots\}$ above, we have the following estimates:
$$t\le\frac{4n}{\log\Delta}\ra\frac{t^2}{n}\le\frac{16n^2}{n\log^2\Delta}=\frac{16n}{\log^2\Delta},$$

$$\frac{4\ov\lambda t}{\Delta}\le\frac{16n}{\log\Delta}\cdot\frac{1}{(\log n)^{2}}\le\frac{16n}{\log\Delta}\cdot\frac{1}{(\log\Delta)^{2}}=o\left(\frac{n}{\log^2\Delta}\right),$$
$$\frac{5n\log\Delta}{\Delta}=o\left(\frac{n}{\log^2\Delta}\right),$$
and thus $\exp\{\ldots\}\le e^{\frac{17n}{\log^2\Delta}}.$
Therefore, from \eqref{fgs-sum} we have 
$$\begin{aligned}
f(G,s)\cdot 2^s&\le\left(\frac{\Delta}{e}\right)^n\sum_{t\le\frac{4n}{\log\Delta}}e^{-\frac{50n}{\log^2\Delta}+\frac{17n}{\log^2\Delta}}\le\left(\frac{\Delta}{e}\right)^ne^{-\frac{20n}{\log^2\Delta}}\sum_{t=1}^ne^{-\frac{13n}{\log^2\Delta}}\\
&\le\left(\frac{\Delta}{e}\right)^ne^{-\frac{20n}{\log^2\Delta}}\cdot o\left(\frac{1}{n}\right).\end{aligned}$$
Hence $\sum_{s>s^*}f(G,s)\cdot 2^s=o\left(\left(\frac{\Delta}{e}\right)^ne^{-\frac{20n}{\log^2\Delta}}\right)$, which implies by \eqref{permanentsum} that
\begin{equation} \label{negligible}
\sum_{s\le s^*}f(G,s)\ge\frac{1-o(1)}{2^{s^*}}\left(\frac{\Delta}{e}\right)^ne^{-\frac{20n}{\log^2\Delta}}\ge\frac{1-o(1)}{2^{s^*}}\left(\frac{\Delta}{e}\right)^ne^{-o(n)}.
\end{equation}

Moving to the third and fourth steps, let $F$ be a 2-factor with $s\le s^*$ cycles. We will show how $F$ can be turned into a Hamilton cycle after $O\left(s\cdot\frac{\log n}{\log\frac{d}{\ov\lambda}}\right)$ edge replacements. We begin by letting $C$ be a cycle in $F$. Since $G$ is connected, some vertex $v$ in $C$ has a neighbor in some cycle $C'\ne C$ of $F$. This implies that with one edge replacement we can convert $C\cup C'$ into a path $P$ which uses the vertices of $C\cup C'$. Now by \autoref{pathlemma}, we can perform $O\left(\frac{\log n}{\log\frac{d}{\ov\lambda}}\right)$ edge replacements to convert $P$ to a cycle or extend to a vertex outside $C\cup C'$; in either case we can then use at most one edge replacement to obtain a path $P'$ which uses the vertices of $C\cup C'\cup C''$ for some other cycle $C''$ in $F$. Continuing in this way, we see that we use $O\left(\frac{\log n}{\log\frac{d}{\ov\lambda}}\right)$ edge replacements to consume a cycle at each step, so after $O\left(s\cdot\frac{\log n}{\log\frac{d}{\ov\lambda}}\right)$ edge replacements we have a Hamilton cycle.

 On the other hand, for a given Hamilton cycle the number of 2-factors that can be converted to it in at most $k$ edge replacements is at most ${n \choose k}\Delta^{2k}$. Hence
$$\sum_{s\le s^*}f(G,s)\le h(G){n\choose k}\Delta^{2k}.$$
for $k=O\left(s^*\cdot\frac{\log n}{\log\frac{d}{\ov\lambda}}\right)$. 

We may now complete the proof by combining all of the bounds we have collected. From equation \eqref{negligible} we get
$$h(G)\ge\left(\frac{\Delta}{e}\right)^n\cdot\frac{1-o(1)}{2^{s^*}{n\choose k}\Delta^{2k}}\cdot e^{-o(n)}.$$
Now since 
$$k=O\left(s^*\cdot\frac{\log n}{\log\frac{d}{\ov\lambda}}\right)=O\left(\frac{n\log n}{\log^2\Delta\log(d/\ov\lambda)}\right)\le O\left(\frac{n}{\log\Delta}\right),$$ for some constant $C$ we have 
$${n\choose k}\le{n\choose Cn/\log\Delta}\le\left(\frac{en}{Cn/\log\Delta}\right)^\frac{Cn}{\log\Delta}\le e^\frac{Cn\log\log\Delta}{\log\Delta}=e^{o(n)}$$
and by (3$'$),
$$\Delta^{2k}\le e^{O\left(\frac{n}{\log\Delta}\frac{20\log n}{\log\frac{d}{\ov\lambda}}\right)}=e^{o(n)}.$$
Therefore
$$h(G)\ge\left(\frac{\Delta}{e}\right)^ne^{-o(n)}\ge n!\left(\frac{d}{n}\right)^n(1-o(1))^n.$$

\subsection{Bipartite case}

Using the corollaries of the bipartite expander mixing lemma included in the appendix, one can follow the proof of \autoref{krivelevich-result} in \cite{Krivelevich2006} with only two other modifications:
\begin{itemize}
    \item When we apply the expander mixing lemma to bound the average degree of $G[V-V_0]$, we use \autoref{bipartite-cor-1} instead.
    \item When we combine the cycles of a 2-factor to form a Hamilton cycle, we must check that the odd-length condition of \autoref{bipartite-path} is satisfied each time we use it. The condition is indeed met because the paths to which we apply \autoref{bipartite-path} have as vertex set the union of disjoint cycles in a bipartite graph.
\end{itemize}

\section{$G$-creating Hamilton paths} \label{main result section}

\subsection{$C_6$, $C_8$ and $C_{10}$}

We begin with claim (a) of \autoref{mainresult}. As noted in Section 2, the graphs $X_n$ satisfy (5) and (6), so by \autoref{bipartite-version}, we have $h(X_n)\ge\left(\frac{q}{e}\right)^ne^{-o(n)}\ge n^{\frac{1}{3}n-o(n)}$. Moreover, $X_n$ is $C_6-$free (see Section 3 of \cite{LAZEBNIK1999503}). Thus if $n=2(q^3+q^2+q+1)$ for some prime power $q$ then $\ov{H_n(C_6)}\ge n^{\frac{1}{3}n-o(n)}$ and \autoref{gfreegraphs} gives the upper bound $H_n(C_6)\le n^{\frac{2}{3}n+o(n)}$ for each such $n$. A density of primes argument similar to the one given in Section 3 of \cite{new_bounds} suffices to extend this to general $n$.

The proof of claim (b) is similar: using \autoref{bipartite-version}, we have $h(Y_n)\ge n^{\frac{1}{5}n-o(n)}$ if $n=2(q^5+q^4+q^3+q^2+q+1)$ for a prime power $q$ and a density of primes argument extends this to general $n$.

For (c), we use the fact that $Y_n$ is also $C_8-$free (see Section 3 of \cite{LAZEBNIK1999503}).

\subsection{Longer even cycles}

Let $k$ be given. Note that if $p$ is large enough, then if $q>p^{k/2+\delta}$ for some $\delta>0$, we have $g(X^{p,q})>2k+4\delta-\log_p4>2k$. We therefore need the following fact.
\begin{claim}
    Let $\varepsilon>0$. Then for all sufficiently large $n>0$, there exists $m\in(n,(1+\varepsilon)n)$ and primes $p,q\equiv 1\pmod 4$ with $\left(\frac{p}{q}\right)=-1$, satisfying $q\in(p^{k/2+\delta},(1+\varepsilon)p^{k/2+\delta})$ and $m=q(q^2-1)$.
\end{claim}
The analagous statement for $\left(\frac{p}{q}\right)=1$ was proven in Lemma 4.3 of \cite{new_bounds}, and we refer to that paper for the proof since the substitution of nonresidues for residues does not affect the argument.

By applying \autoref{bipartite-version} to the graphs $X^{p,q}$, this shows that for large enough $n\in\NN$, there exists $m\in[n,(1+\varepsilon)n]$ such that 
$$\ov{H_m(C_{2k})}\ge m!\left(\frac{p+1}{m}\right)^m(1+o(1))^m$$
and so
$${H_n(C_{2k})}\le H_m(C_{2k})\le\left(\frac{m}{p+1}\right)^m(1+o(1))^m.$$
Noting that
$$m\le q^3\le(1+\varepsilon)^3p^{\frac{3k}{2}+3\delta}$$
this shows
$$\ov{H_n(C_{2k})}\le(1+\varepsilon)^\frac{2m}{k+\delta/2}m^{\left(1-\frac{2}{3k+3\delta/2}\right)m+o(m)}\le n^{\left(1-\frac{2}{2k+3\delta/2}\right)n+\varepsilon n+o(n)}.$$
Taking $\delta,\varepsilon\to 0$ proves the claim.

\subsection{$K_{2,3}$ and $K_{2,4}$}

The graphs $Z_n$ have average degree $d=q(1+o(1))=(2n)^{1/2}(1+o(1))$ so \autoref{hamiltonresult} gives $h(W_n)=n!\left(\frac{2^{1/2}n^{1/2}}{n}\right)^n(1+o(1))^n=\frac{n!}{2^{n/2}n^{n/2}}(1+o(1))^n$ and then \autoref{gfreegraphs} gives $H_n(K_{2,3})\le n^{n/2}2^{-n/2+o(n)}$.

The graphs $W_n$ have average degree $d=(3n)^{1/2}(1+o(1))$ so \autoref{hamiltonresult} gives $h(W_n)=n!\left(\frac{3^{1/2}n^{1/2}}{n}\right)^n(1+o(1))^n$ and then \autoref{gfreegraphs} gives $H_n(K_{2,4})\le n^{n/2}3^{-n/2+o(n)}.$ 

As noted in \cite{FUREDI1996141}, the prime power $q$ required to define these graphs exists in the interval $(\sqrt nt-n^{1/3},\sqrt nt)$, for sufficiently large $n$. This is a dense enough set of integers to extend the bounds given above to all integers $n$.

\section{A lower bound for $K_{2,3}$} \label{lower lound section}
In Theorem 6.1 of \cite{KM} it is shown that for infinitely many $n$ there is a family of at least $((1+\sqrt 5)/2)^{n-o(n)}$ pairwise \textit{colliding} permutations, i.e. there is some $i$ such that $|\tau(i)-\sigma(i)|=1$. {Let $\s F$ be the set of inverses of such a family, so that for any $\sigma,\tau\in\s F$ there is some $i$ such that $\sigma(i)=\tau(i+1)$ or $\tau(i)=\sigma(i+1)$. We describe a family of $K_{2,3}-$creating Hamilton paths in $K_{4n}$. Let the vertices of $K_{4n}$ be labeled $a_1,b_1,x_1,y_1,\ldots,a_n,b_n,x_n,y_n$. To $\sigma\in\s F$ we associate the path
$$P_\sigma=a_1,x_{\sigma(1)},b_1,y_{\sigma(1)},\ldots,a_n,x_{\sigma(n)},b_n,y_{\sigma(n)}.$$
If say $\sigma(i)=\tau(i+1)$, then $P_\sigma\cup P_\tau$ contains a copy of $K_{2,3}$ on the vertices $x_{\tau(i+1)},y_{\tau(i+1)},b_i,a_{i+1},b_{i+1}$. Therefore for such $n$ we have
$$H_{4n}(K_{2,3})\ge\left(\frac{1+\sqrt 5}{2}\right)^{n-o(n)}.$$}

\section{Concluding remarks} \label{conclusion section}

In \autoref{mainresult} and \autoref{lps-improvement} we improved the best known upper bounds on $H_n(C_{2k})$ for every $k\ge 3$. Here we discuss a possible approach to obtaining a further improvement for every $k\ge 6$. In \autoref{lps-improvement} we use the graphs of Lubotzky, Phillips and Sarnak because they are dense graphs with high girth, but a denser construction is known: these are the graphs $CD(k,q)$ of Lazebnik, Ustimenko, and Woldar \cite{LUW}. For primes powers $q$ there exists a graph $CD(k,q)$ which is $q-$regular and bipartite with $2q^{k-\lfloor\frac{k+2}{4}\rfloor+1}$ vertices and with girth at least $k+5$, if $k$ is odd. The additional improvement on $H_n(C_{k+5})$ would follow from the following open conjecture of Ustimenko (for background see \cite{MOORHOUSE20171}).

\begin{conjecture}[Ustimenko] \label{ustimenko-conjecture}
For all $(k,q)$, the nontrivial eigenvalues of $D(k,q)$ (and therefore $CD(k,q)$) are at most $2\sqrt q$.
\end{conjecture}

Currently this conjecture is known for $k=2$ and $k=3$ (Li-Lu-Wang \cite{LI2009}), $k=4$ (Moorhouse-Sun-Williford \cite{MOORHOUSE20171}), and $k=5$ (Gupta-Taranchuk \cite{eigenvalues}). Moreover $CD(6,q)$ is isomorphic to $D(5,q)$ and so the bound also holds for $CD(6,q)$.

For any $k$ for which \autoref{ustimenko-conjecture} holds, \autoref{bipartite-version} gives $h(CD(k,q))\ge n^{\frac{1}{k-\lfloor(k+2)/4\rfloor+1} n-o(n)}$ and \autoref{gfreegraphs} gives $H_n(C_{2k})\le n^{\left(1-\frac{1}{2k-2-\lfloor(2k)/4\rfloor+1}\right)n+o(n)}$ which improves the exponent in \autoref{lps-improvement}. In light of the known cases of the conjecture mentioned above, the graphs $CD(k,q)$ obtain this bound for $H_n(C_{2k})$ when $k=2,3,4$. This matches the best known bounds for $C_4$ and $C_8$, but good bounds on the eigenvalues of $CD(k,q)$ would have to be proved for some $k\ge 10$ before new improvements could be obtained. Since the pseudorandomess requirement of \autoref{hamiltonresult} is considerably weaker than the conjecture $\lambda\le2\sqrt q$, a weaker bound such as $\lambda\le q^{1-\varepsilon}$ would suffice, and moreover the bound would not have to be verified for all prime powers $q$ but only for all $q$ in some suitably dense subset of $\NN$.

Regarding the gap between the lower and upper bounds for $H_n(C_{2k})$ ($k\ge 3$), we do not have a guess as to which quantity is closer to the truth. It is not immediately clear whether the upper bounds coming from \autoref{gfreegraphs} are asymptotically sharp, even if it is applied to optimal $G-$free graphs. Moreover, it is not clear whether the best construction for a lower bound on $\ov{H_n(G)}$ is obtained by taking all Hamilton cycles in a $G-$free graph; Harcos and Solt\'esz noted in \cite{new_bounds} that this is the case if $G=C_3,C_4,C_5$, but not if $G=K_{3,3}$. We were unable to find a construction to show this is not the case for $K_{2,3}$ and $K_{2,4}$. Rather than determine the asymptotics of $H_n(C_{2k})$ for every $k$, a more modest question is perhaps: does this quantity increase or decrease with $k$?

In the case of $H_n(K_{2,4})$ we displayed the exponential factor in \autoref{complete bipartite results} because existing bounds on $H_n(K_{1,4})$ already imply that $H_n(K_{2,4})\le n^{n/2+o(n)}$. The latter quantity has been determined up to an exponential factor:
$$n^{\frac{1}{2}n}2^{-1.47n-o(n)}\le H_n(K_{1,4})\le n^{\frac{1}{2}n}2^{-0.72n+o(n)}$$
(see \cite{KMo} for the lower bound and \cite{KMu} for the upper bound). Our upper bound on $H_n(K_{2,4})$ is a slight improvement for $K_{2,4}$ since $3^{-n/2}<2^{-0.79n}.$ However we expect that the true value of $H_n(K_{2,4})$ is much smaller, given the following conjecture of Harcos and Solt\'esz:

\begin{conjecture}[\cite{new_bounds}] \label{K24 conjecture}
There is a constant $C>0$ such that $H_n(K_{2,4})\le C^n$.
\end{conjecture}

In \cite{new_bounds} it was shown that \autoref{K24 conjecture} is equivalent to the reversing permutations conjecture. Two permutations of $[n]$ are called \textit{reversing} if there are two coordinates where they have the same entries but in the opposite order, and the reversing permutations conjecture states that the maximum size of a family of pairwise reversing permutations of $n$, denoted $RP(n)$, is at most $C^n$ for some constant $C>0$. The best lower bound for $H_n(K_{2,4})$ we are aware of comes from this connection and it shows that $H_{4n}(K_{2,4})\ge RP(n)\ge 1.587^{n-2}$ using a construction also given in \cite{new_bounds}.

It is perhaps more promising to attempt to improve on the construction given in Section \ref{lower lound section} for $H_n(K_{2,3})$. That construction is only a slight improvement over the one implied by the lower bound for $H_n(K_{2,4})$, since the base is increased from $1.587^{1/4}$ to $1.618^{1/4}$.

We now consider some possible improvements of our results counting Hamilton cycles, focusing on \autoref{hamiltonresult}. A common problem in graph theory is to obtain the weakest possible assumptions on a graph that guarantee the existence of a Hamilton cycle. A longstanding conjecture from \cite{Krivelevich2003} is that there is a constant $C$ such that any $(n,d,\lambda)-$graph with $d/\lambda>C$ is Hamiltonian. Recently, Glock, Correia, and Sudakov in \cite{glock2023hamilton} proved that the conjecture holds for graphs where $d$ is polynomial in $n$, and that for general $(n,d,\lambda)-$graphs the condition $d/\lambda\ge C(\log n)^{1/3}$ guarantees Hamiltonicity. In \cite{glock2023hamilton}, Lemma \ref{pathlemma} (Lemma 3.1 of \cite{Krivelevich2012}) is improved, and it would be interesting to see if combining the methods in this paper with \cite{glock2023hamilton} could give any improvements. When we care about the number of Hamilton cycles, there are several ways one could formulate an analogous question. Perhaps most naturally, we would ask how much we can weaken conditions (1)-(3) such that the formula $h(G)=n!(d/n)^n(1+o(1))^n$ still holds, since this represents similar behavior to the Erd\H{o}s-R\'enyi random graph $G(n,d/n)$. We suspect that it is possible to decrease the right-hand side of (2) to $(\log n)^{1+\varepsilon}$ for an arbitrary constant $\varepsilon>0$, which is the corresponding condition appearing in \cite{Krivelevich2012}. The only place we needed the increased exponent 2 is in equation \eqref{permanent-estimate} (and there we could have instead used $d/\lambda\gg(\log\Delta)^2$).

 Another natural question is whether the condition (1) on the degrees in the graph can be relaxed, perhaps to a condition of the form $\Delta/\delta\le 1+\varepsilon$ for some positive $\varepsilon>0$, such that the conclusion $h(G)=n!(d/n)^n(1+o(1))^n$ still holds. We believe that the techniques in this paper still can be used to prove nontrivial bounds on $h(G)$ under weaker assumptions about the degrees, perhaps using stronger pseudorandomness conditions. However, the resulting lower bound may be less than $n!(d/n)^n(1+o(1))^n$, and when the degrees may differ by a constant factor, it is unclear what the desired formula for $h(G)$ should be. We note some places in our proof where we have used bounds on the degrees:
 \begin{itemize}
     \item in \autoref{eml}, to bound the degree variance;
     \item in \autoref{emlcor3};
     \item in the bounds on $g(n/2,z)$ in \autoref{doubly-superstochastic};
     \item in equation \eqref{permanent-estimate};
     \item in estimating the average degree of $G[V-V_0]$ in \autoref{counting cycles section}.
 \end{itemize}

 For irregular graphs one often considers the eigenvalues of the combinatorial Laplacian $L$ rather than adjacency matrix. In \cite{Butler2010}, Butler and Chung built on \cite{Krivelevich2003} and gave a mild condition on the spectral gap of $L$ which guarantees Hamiltonicity. It would be interesting to prove a counting version of \cite{Butler2010}, though again without some condition on the degrees of the graph it is not even clear how many Hamilton cycles one should expect. 

\section{Acknowledgements} 
The authors would like to thank Jason Williford for pointing them to reference \cite{BCN} for the calculation of the eigenvalues of the point graph of a generalized polygon. The authors would like to thank Felix Lazebnik and Vladislav Taranchuk for their useful comments about the graphs $D(k,q)$.

\printbibliography

\section*{Appendix}

Below are the corollaries of the expander mixing lemma analogous to those used in \cite{Krivelevich2003} along with the lemma analogous to that used in \cite{Krivelevich2012}. Where we do not give a proof, it is the same as in \cite{Krivelevich2003}, and where we do, it is a small modification of the proof given in that paper.

\subsection{Corollaries of the expander mixing lemma: irregular case}

\begin{corollary}
\label{emlcor1}
For every subset $U\subseteq V$, 
$$\left|e(U)-d\frac{|U|^2}{2n}\right|\le\frac{\ov\lambda|U|}{2}.$$
\end{corollary}

\begin{corollary}
\label{emlcor2}
Every subset $U\subseteq V$ of cardinality $|U|\le\frac{\ov\lambda n}{d}$ spans at most $\ov\lambda|U|$ edges.
\end{corollary}

\begin{corollary}
\label{emlcor3}
For every subset $U\subseteq V$ of cardinality $1\le|U|\le\frac{2\ov\lambda^2n}{d^2}$, we have
$$|N(U)|>\frac{(d-2\ov\lambda)^2}{5\ov\lambda^2}|U|.$$
\end{corollary}
\begin{proof}
Denote $W:=N(U)$. By \autoref{emlcor2}, $e(U)\le\ov\lambda|U|$. As the degree of every vertex is $U$ is at least $\delta$, we obtain
${e(U,W)\ge\delta|U|-2e(U)\ge\delta|U|-2\ov\lambda|U|}.$ On the other hand, from \autoref{eml} we have ${e(U,W)\le\frac{|U||W|d}{n}+\ov\lambda\sqrt{|U||W|}}.$ Thus
$$
\frac{|U||W|d}{n}+\ov\lambda\sqrt{|U||W|}\ge(\delta-2\ov\lambda)|U|.
$$
However if $|W|\le\frac{(d-2\ov\lambda)^2}{5\ov\lambda^2}|U|$, then we have
$$\frac{|U||W|d}{n}+\ov\lambda\sqrt{|U||W|}<\frac{2(d-2\ov\lambda)|U|}{5}+\frac{(d-2\ov\lambda)|U|}{\sqrt 5}=\gamma(d-2\ov\lambda)|U|$$
where $\gamma=\frac{2}{5}+\frac{1}{\sqrt 5}<1$. Choose $\eta\in(\gamma,1)$. Then using (1$'$) and $\ov\lambda\ll d$, we have
$$d\eta|U|<(d-3\ov\lambda)|U|<(\delta-2\ov\lambda)|U|<(d-2\ov\lambda)\gamma|U|<d\gamma|U|$$
which is a contradiction.
\end{proof}

\begin{corollary}
\label{emlcor4}
For every subset $U\subseteq V$ of cardinality $|U|>\frac{2\ov\lambda^2n}{d^2}$, we have $|N(U)|>\frac{n}{2}-|U|$.
\end{corollary}
\begin{proof}
Set $W=V-(U\cup N(U))$. Then $e(U,W)=0$. On the other hand, ${e(U,W)\ge\frac{|U||W|d}{n}-\ov\lambda\sqrt{|U||W|}}.$ Therefore $\frac{|U||W|d}{n}\le\ov\lambda\sqrt{|U||W|}$ implying
$$|W|\le\frac{\ov\lambda^2n^2}{d^2|U|}<\frac{n}{2}.$$
Hence $|N(U)|=|V|-|U|-|W|>\frac{n}{2}-|U|$.
\end{proof}

\begin{corollary}
\label{emlcor5}
If disjoint subsets $U,W\subseteq V$ are not connected by an edge, then $|U||W|<\frac{\ov\lambda^2n^2}{d^2}$.
\end{corollary}

\begin{corollary}
\label{emlcor6}
$G$ is connected.
\end{corollary}

\lemmatwo*
\begin{proof} Let $P=v_1,\ldots,v_l$ be a path in $G$. If at any step in the argument some path $P^*$ (one of the rotations of $P$ we will describe below) is not maximal, then $P^*$ satisfies the conclusion of \autoref{pathlemma}. Thus, we may assume that all relevant paths are maximal. Using Posa rotation, for $v_i\sim v_\ell$, the path $P' = v_1,\ldots, v_i, v_\ell, v_{\ell+1},\cdots, v_{i+1}$ is also maximal. We say that $P'$ is a {\em rotation} of $P$ with fixed endpoint $v_1$, pivot $v_i$, and broken edge $v_i\sim v_{i+1}$.

For $t\ge 0$, define $S_t=\{v\in V(P)-\{v_1\}:v$ is the endpoint of a path obtainable from $P$ by at most $t$ rotations with fixed endpoint $v_1$ and all broken edges in $P\}$. Estimating the cardinality of $\{i\ge 2:v_i\in N(S_t),v_{i-1},v_i,v_{i+1}\not\in S_t\}$ in two ways (see Proposition 3.1 of \cite{Krivelevich2003}), we obtain
$$|S_{t+1}|\ge\frac{1}{2}|N(S_t)|-\frac{3}{2}|S_t|.$$
Let
$$\begin{aligned}
t_0&=\left\lceil\frac{\log n}{\log(d/\ov\lambda)}\right\rceil+2;\\
\rho&=2t_0.\end{aligned}$$
Then by \autoref{emlcor3}, as long as $|S_t|\le2\ov\lambda^2n/d^2$ we have $|N(S_t)|\ge(d-2\ov\lambda)^2|S_t|/(5\ov\lambda^2)$ and thus $|S_{t+1}|\ge\frac{1}{2}|N(S_t)|-\frac{3}{2}|S_t|\ge(d-2\ov\lambda)^2|S_t|/(10\ov\lambda^2)-\frac{3}{2}|S_t|$. Thus for all $j\le t$, we have $|S_{j+1}|/|S_j|\ge(d-2\ov\lambda)^2/(10\ov\lambda^2)-3/2\ge(d-2\ov\lambda^2)/(11\ov\lambda^2)$ when $n$ is large enough. Therefore we reach $|S_t|>2\ov\lambda^2n/d^2$ after at most
$$\frac{\log\frac{2\ov\lambda^2n}{d^2}}{\log\frac{(d-2\ov\lambda)^2}{11\ov\lambda^2}}\le t_0-2$$
steps. One additional step together with an application of \autoref{emlcor4} gives us $|S_{t+1}|\ge\frac{n}{4}-2|S_{t+1}|.$ This, by \autoref{emlcor5}, implies that $|N(S_{t+1})|+|S_{t+1}|=n-o(n)$. Applying \autoref{emlcor4} once again, we obtain:
$$|S_{t+2}|\ge\frac{1}{2}|N(S_{t+1})|-\frac{3}{2}|S_{t+1}|\ge\frac{1}{2}(n-o(n))-\frac{5}{2}|S_{t+2}|$$
and therefore $|S_{t+2}|>n/10$. Let $B(v_1):=S_{t_0}$, $A_0=B(v_1)\cup\{v_1\}$. For each $v\in B(v_1)$, we apply the same argument using $v$ as the starting vertex; we conclude that for each $a\in A_0$, there is a set $B(a)$ with at least $n/10$ vertices which are the endpoints of paths starting from $a$ and obtainable from $P$ by at most $\rho$ rotations. Also, $V(P)\ge n/10$. 

We consider the path $P$ to be divided into $2\rho$ segments $I_1,\ldots,I_{2\rho}$, each of length at least $\lfloor V(P)/(2\rho)\rfloor\ge\lfloor n/(20\rho)\rfloor$. Any path $P(a,b)$ obtained as above contains at least $\rho$ of the segments untouched, and we call such segments \textit{unbroken} in $P(a,b)$. We consider sequences $\sigma$ of 2 unbroken segments of $P$ where $\sigma$ specifies the order of the segments along with the direction that each segment is traversed. For a given $\sigma$ we define $L(\sigma)$ to be the set of all pairs $a\in A_0,b\in B(a)$ for which $P(a,b)$ contains $\sigma$. By averaging, there exists some sequence $\sigma_0$ such that $|L(\sigma_0)|\ge\zeta n^2$, where $\zeta:=1/4608$. We define $\hat A:=\{a\in A_0: L(\sigma_0)$ contains at least $\zeta n/2$ pairs with $a$ as the first element$\}$ and for $a\in\hat A$, $\hat B(a):=\{b\in B(a):(a,b)\in L(\sigma_0)\}$; then $|\hat A|,|\hat B(a)|\ge\zeta n/2$. Now let $C_1$ and $C_2$ be the first and second segment, respectively, of $\sigma_0$. Then $|C_i|\ge n/(21\rho)$. Given any path $P'$ and $S\subseteq V(P')$, a vertex $v\in S$ is called an \textit{interior point} of $S$ with respect to $P'$ if the vertices preceding and following $v$ in $P'$ are in $S$; we denote by $int_{P'}(S)$ the set of all such points. Note that if $S\subseteq C_i$ then $int_{P(a,b)}(S)$ is the same for any of the paths $P(a,b)$ containing $\sigma$, so we write $int(S)$ unambiguously.

\begin{claim}\label{segment-interior-claim} There exists $C_i'\subseteq C_i$ with $|int(C_i')|\ge n/(32\rho)$ such that every vertex of $C_i'$ has at least $d/(43\rho)$ neighbors in $int(C_i')$.
\end{claim} 
\begin{proof}
We start with $C_i'=C_i$ and as long as there exists a vertex $v_j\in C_i'$ with fewer than $d/(43\rho)$ neighbors in $int(C_i')$ we delete $v_j$ and repeat. If this procedure is continued for $r=|C_i|/15$ steps then we obtain deleted vertices $R=\{v_1,\ldots,v_r\}$ such that at the $r^{th}$ step,
$$|int(C_i')|\ge|int(C_i)|-3r=(1-o(1))|C_i|-\frac{|C_i|}{5}>\frac{|C_i|}{2}$$
and
$$e(R,int(C_i'))\le\frac{d}{43\rho}\cdot r=\frac{|C_i|d}{43\cdot15\rho}.$$
But by \autoref{eml},
$$\begin{aligned}
e(R,int(C_i'))&\ge\frac{|R||int(C_i')|d}{n}-\ov\lambda\sqrt{|R||int(C_i')|}\\
&\ge\frac{|C_i|}{15}\frac{|C_i|}{2}\frac{d}{n}-\ov\lambda\sqrt{\frac{|C_i|}{15}|C_i|}\\
&=\frac{|C_i|^2d}{2\cdot 15n}-\frac{\ov\lambda|C_i|}{\sqrt{15}}\\
&\ge\frac{n}{21\rho}\frac{|C_i|d}{2\cdot 15n}-\frac{\ov\lambda|C_i|}{\sqrt{15}}.
\end{aligned}$$
Therefore
\begin{equation}\label{claim1eq}
\frac{|C_i|d}{1806\cdot15\rho}=\frac{|C_i|d}{15\rho}\left(\frac{1}{21\cdot 2}-\frac{1}{43}\right)\le\frac{\ov\lambda|C_i|}{\sqrt{15}}.
\end{equation}
On the other hand, since $\rho\le\frac{2\log n}{\log(d/\ov\lambda)}$, the left-hand side above is at least $\frac{|C_i|d\log(d/\ov\lambda)}{1806\cdot 15\cdot 2\log n}$. From (3') it then follows that
$$\frac{|C_i|}{1806\cdot 15\cdot 2}\frac{d\log(d/\ov\lambda)}{\log n}>\frac{|C_i|}{1806\cdot15\cdot 2}\frac{1806\cdot 5\cdot 2}{\sqrt{15}}\ov\lambda$$
which contradicts \eqref{claim1eq}. Therefore the procedure stops before the $\frac{|C_i|}{15}^{th}$ step, and we end with a set $C_i'$ which satisfies
$$\begin{aligned}
|int(C_i')|&\ge|int(C_i)|-3r=(1-o(1))|C_i|-\frac{|C_i|}{5}>\frac{n}{32\rho}.\end{aligned}$$
\end{proof}

We fix sets $C_1'$ and $C_2'$ obtained as above. Again from (3') we have
$$\begin{aligned}
|\hat A||int(C_1')|\ge\frac{\zeta n}{2}\cdot\frac{n}{32\rho}\ge\frac{\zeta}{2\cdot 64}\frac{n^2\log\frac{d}{\ov\lambda}}{\log n}
\ge\frac{n^2\ov\lambda}{d}\ge\frac{n^2\ov\lambda^2}{d^2}\end{aligned}$$
which by \autoref{emlcor5} implies there is a vertex $\hat a\in\hat A$ adjacent to some vertex in $int(C_1')$. Similarly we can show there is a vertex $\hat b\in\hat B(\hat a)$ adjacent to some vertex in $int(C_2')$. Let $x$ be a vertex on $P(\hat a,\hat b)$ between $C_1'$ and $C_2'$. Let $P_1$ be the subpath of $P$ from $\hat a$ to $x$ and let $P_2$ be the subpath of $P$ from $x$ to $\hat b$. Consider $P_1$. For $i\ge 0$, let $T_i:=\{v\in C_1'\sm\{x\}: v$ is the endpoint of a path obtainable from $P_1$ by at most $i$ rotations with fixed endpoints $x$, all pivots in $int(C_1')$ and all broken edges in $P_1\}$.  Similarly to the sets $S_t$ considered above, we can show
\begin{equation}\label{ti-equation}
|T_i|\ge\frac{1}{2}|N(T_i)\cap int(C_1')|-\frac{3}{2}|T_i|.
\end{equation}
\begin{claim} \label{rotations-claim} For each $U\subseteq C_1'$ with $|U|\le\ov\lambda n/d$ we have $|N(U)\cap int(C_1')|\ge \beta d^2|U|/(\rho^2\ov\lambda^2)$, where $\beta=(1/86)^2$.
\end{claim}
\begin{proof}
Since $U\subseteq C_1'$, every vertex $u\in U$ has at least $d/(43\rho)$ neighbors in $int(C_1')$. Therefore $e(U,int(C_1'))\ge d|U|/(43\rho)$. Set $W=N(U)\cap int(C_1')$. If $|W|<\beta d^2|U|/(\rho^2\ov\lambda^2)$, then by \autoref{eml} we have
$$\begin{aligned}
\frac{d|U|}{43\rho}\le e(U,W)\le\frac{|U||W|}{n}+\ov\lambda\sqrt{|U||W|}
&<\frac{\ov\lambda n}{d}\frac{|U|\beta d^2}{\rho^2\ov\lambda^2n}+\frac{\sqrt \beta d}{\rho}|U|
<\left(\frac{\beta d}{\rho}+\frac{\sqrt {\beta}d}{\rho}\right)|U|<\frac{2\sqrt\beta d}{\rho}|U|.\end{aligned}$$
By the choice of $\beta$, this is a contradiction.
\end{proof}
Since $\hat a$ has a neighbor in $int(C_1')$, we have $|T_1|\ge 1$. Now combining equation \eqref{ti-equation} with Claim 2 and the inequality $\rho\ll d/\ov\lambda$, we can show that as long as $|T_i|\le\frac{\ov\lambda n}{d}$, we have $|T_{i+1}|\ge\frac{d}{\ov\lambda}|T_i|$. Thus, using (2$'$), we reach $|T_{i+1}|\ge\frac{\ov\lambda n}{d}$ after at most $\frac{\log n}{\log\frac{d}{\ov\lambda}}$ rotations. The same is true when we consider $P_2$, and so we then apply \autoref{emlcor5} to obtain a Hamilton cycle after a total of $O\left(\frac{\log n}{\log\frac{\ov d}{\ov\lambda}}\right)$ rotations.
\end{proof}

\subsection{Corollaries of the expander mixing lemma: bipartite case}

\begin{corollary}\label{bipartite-cor-1}
    For any $U\subseteq V$, $e(U)\le\frac{d|U|^2}{2n}+\frac{\lambda|U|}{2}$.
\end{corollary}
\begin{proof}
    Let $U=S\cup T$ where $S\subseteq X$ and $T\subseteq Y$. Then:
    $$\begin{aligned}
        e(U)=e(S,T)&\le\frac{2d|S||T|}{2n}+\lambda\sqrt{|S||T|}\\
        &\le\frac{d(|S|+|T|)^2}{2n}+\lambda\frac{|S|+|T|}{2}\\
        &=\frac{d|U|^2}{2n}+\frac{\lambda|U|}{2}.
    \end{aligned}$$
\end{proof}

\begin{corollary}\label{bipartite-cor-2}
    If $|U|\le\frac{\lambda n}{d}$ then $e(U)\le\lambda|U|$.
\end{corollary}
\begin{proof}
    This follows immediately from \autoref{bipartite-cor-1}.
\end{proof}

\begin{corollary}\label{bipartite-cor-3}
    If $|U|<\frac{\lambda^2 n}{d^2}$ then $|N(U)|>\frac{(d-2\lambda)^2}{4\lambda^2}|U|$.
\end{corollary}
\begin{proof}
    Let $W=S\cup T$ with $S\subseteq X$, $T\subseteq Y$. Note
    $$e(U,W)=d|U|-2e(U)\ge(d-2\lambda)|U|.$$
    Also
    $$\begin{aligned}
        e(U,W)&=e(S,W)+e(T,W)\\
        &\le\frac{2d|S||W|}{n}+\lambda\sqrt{|S||W|}+\frac{2d|T||W|}{n}+\lambda\sqrt{|T||W|}\\
        &\le\frac{2d|U||W|}{n}+\lambda\sqrt{|U||W|}.
    \end{aligned}$$
This gives
$$\frac{2d|U||W|}{n}+\lambda\sqrt{|U||W|}\ge(d-2\lambda)|U|.$$
If $|W|\le\frac{(d-2\lambda)^2}{4\lambda^2}|U|$ then the left-hand side above is at most
$$\begin{aligned}
    &\frac{2d(d-2\lambda)^2|U|^2}{4\lambda^2n}+\frac{\lambda|U|(d-2\lambda)}{4\lambda}\\
    <&\left(\frac{2}{4}+\frac{1}{\sqrt 4}\right)(d-2\lambda)|U|=(d-2\lambda)|U|
\end{aligned}$$
which is a contradiction.
\end{proof}

\begin{corollary}\label{bipartite-cor-4}
    If $S\subseteq X$ or $S\subseteq Y$, and $|S|>\frac{\lambda n^2}{d^2}$, then $|N(S)|>\frac{n}{4}.$
\end{corollary}
\begin{proof}
Assume $S\subseteq X$. Set $W=Y-N(S)$. Then $0=e(S,Y)\ge\frac{2d|S||W|}{n}-\lambda\sqrt{|S||W|}$. Rearranging, we find
$$|W|\le\frac{\lambda^2n^2}{4d^2|S|}<\frac{n}{4}.$$
Hence $|N(S)|=|Y|-|W|\ge\frac{n}{2}-\frac{n}{4}=\frac{n}{2}$.
\end{proof}

\begin{corollary}\label{bipartite-cor-5}
If $S\subseteq X,T\subseteq Y$ are not connected by an edge, then $|S||Y|<\frac{\lambda^2n^2}{4d^2}$.
\end{corollary}
\begin{proof}
    This follows immediately from \autoref{bipartite-eml}.
\end{proof}

\begin{corollary}\label{bipartite-cor-8}
$G$ is connected.
\end{corollary}

\bipartitelemmatwo*
\begin{proof}
The proof is similar to that of \autoref{pathlemma}, with some changes in the values of numerical constants. We list the main modifications to the argument that are needed. Let $X,Y$ be the partite sets of $G$.
\begin{itemize}
    \item We need the fact that the endpoints obtained by rotating $P$ with the same fixed endpoint are all in the same partite set, say $X$ (i.e. for every $t$, $S_t\subseteq X$).
    \item We do not include $v_1$ in $A_0$ as $v_1$ is possibly in a different partite set from the other vertices in $A_0\subseteq X$. Since $P$ has odd length we thus have $A_0\subseteq X$ and $B(a)\subseteq Y$ for every $a\in A$.
    \item Instead of proving \autoref{segment-interior-claim}, we show there exists $C_1'\subseteq C_1$ with $|int(C_1')\cap Y|\ge n/32\rho$ such that every vertex in $C_1'\cap X$ has at least $d/43\rho$ neighbors in $int(C_1')\cap Y$. The proof is similar except that we only delete vertices in $C_1'\cap X$, and  $int(C_1')\cap Y$ plays the role of $int(C_1')$. (Similarly we show there exists $C_2'\subseteq C_2$ with $|int(C_2')\cap X|\ge n/32\rho$ such that every vertex in $C_2'\cap Y$ has at least $d/43\rho$ neighbors in $int(C_2')\cap X$.)
    \item We need the fact that $T_i\subseteq X$ and the analogously defined set of endpoints corresponding to $C_2$ is a subset of $Y$. Thus when we show that the two sets of endpoints are large enough we can find an edge between them.
\end{itemize}
\end{proof}

\end{document}